\documentclass[12pt,amscd]{amsart}
\usepackage[all]{xy}
\usepackage{graphicx}
\usepackage{amsmath,amsxtra,amssymb,latexsym, amscd,amsthm}
\usepackage{indentfirst}
\usepackage[mathscr]{eucal}
\usepackage[T1]{fontenc}

\newtheorem{thm}{Theorem}[section]
\newtheorem{cor}[thm]{Corollary}
\newtheorem{lem}[thm]{Lemma}

\theoremstyle{definition}

\newtheorem{rem}[thm]{Remark}

\DeclareMathOperator{\N}{\mathbb {N}}
\DeclareMathOperator{\Z}{\mathbb {Z}}
\DeclareMathOperator{\R}{\mathbb {R}}
\DeclareMathOperator{\depth}{depth}

\DeclareMathOperator{\dstab}{dstab}

\def\alb {\boldsymbol {\alpha}}
\def\btb {\boldsymbol {\beta}}

\def\gmb {\boldsymbol {\gamma}}

\def\zv {\mathbf 0}

\def\x {\mathbf x}

\def\mi {\mathfrak m}

\def\h {\widetilde{H}}

\def\H {\mathcal{H}}
\def\V {\mathcal{V}}
\def\E {\mathcal{E}}

\begin{document}

\title[Stability of depth functions of cover ideals of balanced hypergraphs] {Stability of depth functions of cover ideals of balanced hypergraphs}

\author{Nguyen Thu Hang}
\address{Thai Nguyen College of Sciences, Thai Nguyen University, Thai Nguyen, Vietnam}
\email{nguyenthuhang0508@gmail.com}

\date{}

\dedicatory{}
\commby{}
%-----------------------------------------------------------
\begin{abstract}We prove that the depth functions of cover ideals of balanced hypergragh have the non-increasing property. Furthermore, we also give a bound for the index of depth stability of these ideals.
\end{abstract}
% -----------------------------------------------------------
\maketitle
% -----------------------------------------------------------
\section*{Introduction}
Let $R=K[x_1,\ldots,x_n]$ be a polynomial ring over a field $K$ and $I\subseteq R$ is a homogenerous ideal. The numerical function $t\longmapsto \depth R/I^t$, $t\geqslant 1$, is called the {\it depth function} of $I.$ This function has been studied by many authors (see e.g. \cite{HHTT, HTT, HH, HQ, HV, NV,TNT}). It is the well-known result by Brodmann \cite{B} that this function is constant when $t$ large enough. Moreover, Brodmann also \cite{B} proved that:
$$\underset{t\rightarrow \infty}{\lim} \depth R/I^t\leqslant \dim R-\ell(I),$$
where $\ell(I)$ is the analytic spread of $I$. Eisenbud and Huneke \cite{EH} showed that equality hold when the associated graded ring of $I$ is Cohen-Macaulay. 

The smallest number integer $s$ such that $\depth R/I^t=\depth R/I^s$ for all $t\geqslant s$ is called the {\it index of depth stability} and denoted by $\dstab(I).$ It is of natural interest to find a bound for $\dstab(I).$ For until now, there are only a few classes of ideals which are able to find a bound for $\dstab(I),$ for instance \cite {CN, NT,HH,HQ,TNT}.

In general, the behavior of depth functions of monomial ideals is complicated, it is proved in \cite{HHTT} that any convergent non-negative numerical function is the depth function of powers of a monomial ideal. Herzog and Hibi \cite{HH}  asked whether depth function is non-increasing for any square-free monomial ideals. However, Kaiser, Stehl${\rm \acute{i}}$k and ${\rm \check{S}}$krekovski \cite{KSS} gave a counterexample to show that there is a graph whose cover ideal has not non-increasing depth function. Until now, this problem can be only characterized on some certain classes of monomial ideals, for example  \cite{CPSTY,HTT,NT,HV,HKTT,NV}. 

This work is motivated by a recent paper of me and Trung \cite{NT}. In that paper we proved that the depth functions of cover ideals of unimodular hypergraphs are non-increasing and gave a reasonable bound for the index of depth stability of these ideals.  In this paper we extend these results for the cover ideals of balanced hypergraghs. The class of balanced hypergraphs contain all unimodular hypergraphs.

Before stating the main results, we recall some basic notations about from graph theory (see \cite{Berge} for more detail). 

Let $\V = \{1,\ldots,n\}$, and let $\E$ be a family of distinct nonempty subsets of $\V$. The pair $\H=(\V,\E)$ is called {\it a hypergraph} with vertex set $\V$ and edge set $\E$. Note that a hypergraph generalizes the classical notion of a graph. It means that a graph is a hypergraph for which every $E \in \E$ has cardinality two.

One may also define a hypergraph by its incidence matrix $A(\H) = (a_{ij})$, with rows representing the edges $E_1,E_2, \ldots,E_m$ and  columns representing the vertices $1,2,\ldots, n$ where $a_{ij} = 0$ if $j\notin E_i$ and  $a_{ij} = 1$ if $j\in E_i$.

A matrix is called balanced matrix if it 
 has no square submatrix of the form
 \[B_k= \begin{pmatrix} 1&1&0& \cdots &0&0 & 0 \\
0&1&1&\cdots &0& 0 &0  \\
0&0&1&\cdots &0&0&0 \\
\vdots &\vdots &\vdots & \cdots & \vdots & \vdots& \vdots\\

0 &0 &0 & \cdots & 0&1 & 1\\
1 &0 &0 & \cdots &0& 0 & 1
\end{pmatrix},\]
 where $k\geq 3$ is odd.
 
A cycle of length $k$ in hypergragh $\H$ is a sequence $(i_1,E_1,i_2,E_2, i_3,\ldots, i_k,E_k,i_1)$ such that: 
\begin{itemize}
\item $E_1,\ldots,E_k$ are distinct edges of $\H$;
\item $i_1,\ldots,i_k$ are distinct vertices of $\H$ such that $\{i_t,i_{t+1}\}\in E_t$ with $t=1,\ldots,k-1$;
\item $\{i_k,i_1\}\in E_k.$
\end{itemize}

 A hypergraph $\H$ is said to be {\it balanced} if every odd cycle has an edge containing at least three vertices of the cycle. In other word, $\H$ is balanced if and only if so is its incidence matrix.
 
A {\it vertex cover} of $\H$ is a subset of $\V$ which meets every edge of $\H$; a vertex cover is {\it minimal} if none of its proper subsets is itself a cover. For a subset  $\tau = \{i_1,\ldots,i_t\}$ of $\V$, set $\x_{\tau} := x_{i_1}\cdots x_{i_t}$. The {\it cover ideal} of $\H$ is then defined by: 
$$J(\H) := (\x_{\tau} \mid \tau \text{ is a minimal vertex cover of } H).$$

It is well-known that there is one-to-one correspondence between squarefree monomial ideals of $R$ and cover ideals of hypergraphs on the vertex set $\V$.

\medskip

Our first main result of this paper is the following theorem.

\medskip

\noindent{\bf Theorem \ref{T1}.} {\it Let $\H$ be a balanced hypergraph. Then the depth function of $J(\H)$ has non-increasing property. 
}

\medskip

We next solve the question of when $\depth R/J(\H)^s$ becomes stationary where $\H$ is a balanced hypergraph, in term of $\dim R$ and the analytic spread of $J(\H)$. Namely:

\medskip

\noindent{\bf Theorem \ref{T2}.} {\it Let $\H = (\V,\E)$ be a balanced hypergraph with the vertex set\break $\V =\{1,\ldots,n\}$. Then
$$\depth R/J(\H)^t = n -\ell(J(\H)) \text{ for all } t\geqslant n.$$
Moreover, $\dstab(J(\H)) \leqslant n$.
}

\medskip

Our approach is based on a generalized Hochster's formula for computing local cohomology modules of arbitrary monomial ideals formulated by Takayama \cite{T}.  By using this formula we are able to investigate the depth of powers of monomial ideals via the integer solutions of certain systems of linear inequalities. This allows us to use the theory of polytopes as the key role in this paper (see e.g. \cite{NT, HKTT, HT} for this approach).

\medskip

The paper is organized as follows. In Section $1$, we set up some basic notations and terminology for simplicial complex, the relationship between simplicial complexes and cover ideals of hypergraphs; and a generalization of Hochster's formula for computing local cohomology modules. In Section $2$, we prove the non-increasing property for depth functions of cover ideals of balanced hypergraphs, we also establish an upper bound for $\dstab(J(\H))$  of any balanced hypergraph $\H$.

\section{Preliminary}
In this section, in order to be convenient we recall some basic notations used in the paper and a number of auxiliary results. Throughout the paper, the important invariant that we investigate are the depth. Though this notion can be defined in several ways, for our purpose we recall their definition by means of local cohomology modules.

Let $R=K[x_1,\ldots,x_n]$ be a polynomial ring over a field $k$ with the maximal homogeneous ideal $\mi=(x_1,\ldots,x_n).$

Let $M$ be a finitely generated graded non-zero $R-$module.
The depth of $M,$ denoted by $\depth M$, is the lenght of any maximal homogeneous $M-$sequence of $M.$ This invariant is one of the most important numerical invariants in commutative algebra. It can be determined via local cohomology of M as follows:
$$\depth (M):= \min\{i\mid H^i_\mi(M)\neq 0\},$$
where $ H^i_\mi(M)$, for $i\geq 0$, denote the $i-$th local cohomology module of $M$ with respect to $\mi.$

Let $I$ be a monomial ideal in $R.$ Since $R/I$ is an $\mathbb N^n-$graded algebra and $H^i_\mi(R/I)$ is an $\mathbb Z^n-$graded module over $R/I$. We denote by $H^i_\mi(R/I)_{\alpha }$ the $\alpha-$component of $H^i_\mi(R/I)$ for some $\alpha =(\alpha _1,\ldots,\alpha _n)\in \mathbb Z^n.$

In order to compute the $i-$th local cohomology module of $R/I$, we use the formula of Takayama \cite{T}.

A {\it simplicial complex} on $[n] = \{1,\ldots, n\}$ is a collection of subsets of $[n]$ such that if $\sigma\in \Delta $ and $\tau\subseteq \sigma$ then $\tau\in \Delta$. The Stanley-Reisner ideal of the simplicial complex $ \Delta $ is defined by 
$$I_{\Delta}:=(x_\tau \mid \tau \notin \Delta) \subseteq R.$$
Note that  if $I$ is a squarefree monomial ideal, then it is a Stanley-Reisner ideal of the simplicial complex $\Delta(I) := \{ \tau \subseteq \V \mid \x_{\tau} \notin I\}$. If $I$ is a monomial ideal (may be not squarefree) we also use $\Delta(I)$ to denote the simplicial complex corresponding to the squarefree monomial ideal $\sqrt{I}$.

Set $\Delta (I)$ is a simplicial complex corresponding to the Stanley-Reisner $\sqrt{I}.$ For any $\alpha =(\alpha _1,\ldots,\alpha _n)\in \mathbb Z^n.$ Set $x^{\alpha }=x_1^{\alpha _1}\cdots x_n^{\alpha _n}$ and $G_{\alpha }=\{i\mid \alpha _i<0\}.$ For every subset $F\subseteq  [n]=\{1,\ldots,n\}$ we let $R_F=R[x_i^{-1}\mid i\in F].$ Define the simplicial complex, denoted $\Delta _{\alpha }(I)$, by
\begin{equation}\label{degree-complex}
\Delta _{\alpha }(I) :=\{F\backslash  G_{\alb}\mid G_{\alb}\subseteq F, x^{\alpha }\notin IR_F\}.
\end{equation}
We call $\Delta _{\alpha }(I)$ to be a {\it degree complex} of $I.$

Takayama's formula is stated as follows:
\begin{equation} \label{TA} \dim_K {H_{\frak m}^i(R/I)_{\alb}}=\dim_K \widetilde{H}_{i-\mid G_{\alpha }\mid-1 }(\Delta _{\alb}(I);K).
\end{equation}

The original formula in \cite[Theorem 1.1]{T}  is bit different. It considers more conditions on $\alb$ for $H_{\frak m}^i(R/I)_{\alb} = 0$. However, the proof in \cite[Theorem 1.1]{T} shows that we might ignore these conditions.

Let $\mathcal{F}(\Delta)$ be the set of facets of simplicial complex $\Delta$. If $\mathcal{F}(\Delta) =\{F_1,\ldots,F_m\}$, we write $\Delta = \left<F_1,\ldots,F_m\right>$.  Then Stanley-Reisner $I_{\Delta}$ has the primary-decomposition (see \cite[Theorem $1.7$]{MS}):
\begin{equation}\label{HT}
I_{\Delta} = \bigcap_{F\in\mathcal F(\Delta)} P_F,
\end{equation}
where  $P_F= (x_i\mid i\notin F)$.

For $s\geqslant 1$, the $s$-th symbolic power of $I_{\Delta}$ is $I_{\Delta}^{(s)} = \bigcap_{F\in\mathcal F(\Delta)} P_F^s.$

Note that $\Delta (I_{\Delta }^{(s)})=\Delta $. There is an useful description of $\Delta _{\alpha }(I^{(n)})$ as follows: 
\begin{lem} \label{MT1}\cite[Lemma 1.3]{MT1} For all $\alpha \in \mathbb N^n$ and $s\geqslant 1$, we have 
$$\Delta _{\alpha }(I^{(s)})= \left<F\in\mathcal F(\Delta)\mid  \sum_{i\notin F} \alpha_i \leqslant s-1 \right>.$$
\end{lem}

Let $\H = (\V, \E)$ be a hypergraph. Then, the cover ideal of $\H$ can be written as
\begin{equation} \label{intersect}
J(\H) = \bigcap_{E\in\E} (x_i\mid i\in E).
\end{equation}
By this formula, combining with Equation $(\ref{HT})$, it is easy to see that $J(\H)$ is a Stanley-Reisner ideal of simplicial complex 
\begin{equation}\label{complex-cover}
\Delta(J(\H)) = \left<\V\setminus E \mid E\in\E\right>.
\end{equation}

In particular, By  \cite[Theorem 1.4]{HHTZ} if $\H$ is a balanced hypergragh then the cover ideal $J(\H)$ is normally torsion-free, i.e., $J(\H)^s=J(\H)^{(s)}$ for all $s\geqslant 1$. Combining with \cite[Lemma $1.3$]{MT1} we obtain:

\begin{lem} \label{uni-complex}Let $\H = (\V,\E)$ be a balanced hypergraph with $\V =\{1,\ldots,n\}$ and $\alb=(\alpha_1,\ldots,\alpha_n)\in\N^n$. Then for every $s\geqslant 1$ we have
$$
\Delta _{\alb}(J(\H)^s)=\left<\V\setminus E\mid E\in \E \text{ and } \sum_{i\in E} \alpha_i \leqslant s-1\right>.
$$
\end{lem}

For $F\subseteq \V$, set $S := [x_i\mid i\notin F]$ and $J' := J(\H)R_F \cap S$. Let $\H'$ be a hypergraph on the vertex set $\V' :=\V\setminus F$ with the edge set $\E' =\{E\in \E \mid E \cap F =\emptyset\}$. By $(\ref{intersect})$ we obtain $J(\H)R_F \cap S = J(\H').$

\section{The behavior of depth functions of cover ideals}

In this section we investigate the non-increasing property of the depth functions of cover ideals of balanced hypergraphs. After that we want to give an effective bound for the index of depth stability of cover ideal $J(\H).$

Let $\H=(\V,\E)$ be a balanced hypergraph on the vertex set $\V =\{1,\ldots,n\}$. Suppose that $\E = \{E_1,\ldots,E_m\}$. Then by Equation $(\ref{complex-cover})$ we have $$\Delta(J(\H)) = \left<\V\setminus E_1,\ldots, \V\setminus E_m\right>.$$

Firstly, let $p\geqslant 0, s\geqslant 1$ and $\alb =(\alpha_1,\ldots,\alpha_n)\in \N^n$ such that $H_{\mi}^p(R/J(\H)^s)_{\alb}\ne\zv$. By Equation $(\ref{TA})$ we have
$$\dim_K \h_{p-1}(\Delta_{\alb}(J(\H)^s);K) =\dim_K H_{\mi}^p(R/J(\H)^s)_{\alb},$$
so that $\h_{p-1}(\Delta_{\alb}(J(\H)^s);K)\ne\zv$. In particular, $\Delta_{\alb}(J(\H)^s)\ne\emptyset$. Thus, by Lemma $\ref{uni-complex}$ we may assume that $\Delta_{\alb}(J(\H)^s) =\left<\V\setminus E_1,\ldots, \V\setminus E_r\right>$ for some $1\leqslant r\leqslant m$. 

For each $t\geqslant 1$, let $\Omega _t \subset \R^n$  be the set of solutions of the following system of linear inequalities:
\begin{equation}\label{EQ-basics}
\begin{cases}
\sum_{i\in E_j} x_i  \leqslant t-1 & \text{ for } j = 1,\ldots,r,\\
\sum_{i\in E_j} x_i  \geqslant  t & \text{ for } j = r+1,\ldots,m,\\
x_1\geqslant 0,\ldots,x_n\geqslant 0.
\end{cases}
\end{equation}
It is obvious $\alb \in \Omega _s$ and so $\Omega _s\neq \emptyset.$ Moreover, by Lemma $\ref{uni-complex}$ if $\beta \in \Omega _t\cap \mathbb N^n$ then 
$$\Delta_{\btb}(J(\H)^t) = \left<\V\setminus E_1,\ldots, \V\setminus E_r\right> = \Delta_{\alb}(J(\H)^s).$$
In order to investigate the set $\Omega _t,$ we can consider $C_t$, which is the set of solutions in $\mathbb R^n$ of the following system of linear inequalities:
\begin{equation}\label{EQ-basics_1}
\begin{cases}
\sum_{i\in E_j} x_i  < t & \text{ for } j = 1,\ldots,r,\\
\sum_{i\in E_j} x_i  \geqslant  t & \text{ for } j = r+1,\ldots,m,\\
x_1\geqslant 0,\ldots,x_n\geqslant 0.
\end{cases}
\end{equation}
Because of for any $t$ we have $\Omega _t\subseteq C_t.$ It deduces $C_s\neq \emptyset$. Since $C_t=tC_1,$ where $C_1$ is the set of solutions of system $(\ref{EQ-basics_1})$ with $t=1.$ It implies $C_1\neq \emptyset.$

Let $\overline{C}_t$ be the closure of $C_t$ in $\R^n$ with respect to the usual Euclidean topology. Then, $\overline{C}_t = t\overline{C}_1$, because of $C_1\ne\emptyset$ it implies $\overline{C}_1\ne\emptyset$. One has $\overline{C}_t$ is solutions set in $\R^n$ of the following system:
\begin{equation}\label{EQ-polytope}
\begin{cases}
\sum_{i\in E_j} x_i  \leqslant t & \text{ for } j = 1,\ldots,r,\\
\sum_{i\in E_j} x_i  \geqslant  t & \text{ for } j = r+1,\ldots,m,\\
x_1\geqslant 0,\ldots,x_n\geqslant 0.
\end{cases}
\end{equation}
It is clear that $\overline{C}_t$ is a convex polyhedron in $\R^n$.

\begin{lem} \label{polytope} $\overline{C}_1$ is a polytope with $\dim \overline{C}_1 = n$. Moreover, every vertex of $\overline{C}_1$ is integer point, i.e., all its coordinates are integers.
\end{lem}
\begin{proof} Firstly, it follows from \cite[Lemma $2.1$]{NT} that $\overline{C_1}$ is a polytope in $\mathbb R^n$ and\break $\dim \overline{C}_1 = n.$

Now, let $\alb =(\alpha_1,\ldots,\alpha_n)$ is a vertex of $\overline{C_1}$. Since $\overline C_1$ is the solutions set in $\R^n$ of the system $(\ref{EQ-polytope})$ with $t=1$, by \cite[Formula $23$ in Page $104$]{S}, one has $\alb$ must be the unique solution of a system of linear equations of the form
\begin{equation}
\begin{cases} \label{CRAMER}
\sum_{i\in E_j} x_i  = 1 & \text{ for } j \in S_1 \subseteq \{1,\ldots,m\},\\
x_j = 0 & \text{ for } j \in S_2 \subseteq \{1,\ldots,n\}
\end{cases}
\end{equation}
where $|S_1| + |S_2| = n$.

Since the matrix of the system: $$\sum_{i\in E_j} x_i  = 1, \ j \in S_1$$ is a submatrix of $A(\H)$. Hence, it is balanced matrix and hence the matrix of the system $(\ref{CRAMER})$ is balanced matrix. By \cite[Theorem 2.17]{S} we have $\alb$ is a $\{0,1\}-$vector. Therefore, $\alb$ is an integer point.
\end{proof}

\begin{rem}\label{R1}Observe that $\Omega _t\subseteq C_t\subseteq \overline C_t$, so $\Omega _t$ is a polytope as well.
\end{rem}
\begin{lem}\label{polytope_1} For any $t\geqslant 1$,  if $\Omega _t \cap \N^n \neq \emptyset$ then $\Omega _{t+1}\cap \N^n \neq \emptyset$.  Moreover $\Omega _{n}\cap \N^n \neq \emptyset$.
\end{lem}
\begin{proof} Let $\alb = (\alpha_1,\ldots,\alpha_n)\in \Omega _t\cap \N^n$ so that from the system $(\ref{EQ-basics})$ we can see $\alb$ satisfies
$$
\begin{cases}
\sum_{i\in E_j} \alpha_i \leqslant t-1 & \text{ for } j = 1,\ldots,r,\\
\sum_{i\in E_j} \alpha_i \geqslant  t & \text{ for } j = r+1,\ldots,m.
\end{cases}
$$
Since $\Omega _t\neq \emptyset$, one has $\overline{C_t}\ne\emptyset$. Therefore, $\overline{C_1}\neq \emptyset$.  Let $\gmb\in \mathbb N^n$ be a vertex of $\overline{C_1}$. Then $\gmb\in\N^n$ by Lemma \ref{polytope}. Note also that $\gmb$  is a solution of the system $(\ref{EQ-polytope})$ with replacing $t$ by $1$.

Let $\theta =\alb+\gmb\in \N^n$. We have
$$
\begin{cases}
\sum_{i\in E_j} \theta _i \leqslant t & \text{ for } j = 1,\ldots,r,\\
\sum_{i\in E_j} \theta _i  \geqslant  t+1 & \text{ for } j = r+1,\ldots,m.
\end{cases}
$$
Thus, $\theta \in\Omega _{t+1}$, and it is clear that $\Omega _{t+1}\cap \N^n\ne\emptyset$, as required.

It remains to show that $\Omega_n \cap \N^n\ne\emptyset$. From the system $(\ref{EQ-basics})$ we can see that $\Omega _n$ is the set of solutions of the following system:
\begin{equation}\label{EQ-Pm}
\begin{cases}
\sum_{i\in E_j} x_i  \leqslant n-1 & \text{ for } j = 1,\ldots,r,\\
\sum_{i\in E_j} x_i  \geqslant  n & \text{ for } j = r+1,\ldots,m,\\
x_1\geqslant 0,\ldots,x_n\geqslant 0.
\end{cases}
\end{equation}

If $r=m$, then the zero vector of $\R^n$ is in $\Omega _n$, and then $\Omega _n\cap\N^n\ne\emptyset$. 

Assume that $r < m$. From the system $(\ref{EQ-polytope})$ we conclude that $\sum_{i\in E_m} x_i  = 1$ is a supporting hyperplane of $\overline C_1$. Let $F$ be the facet of $\overline C_1$ determined by this hyperplane. Now take $n$ vertices of $\overline C_1$ lying in $F$, say $\alb^1, \ldots,\alb^n$, such that they are affinely independent. Let $\alb := (\alb^1+\cdots+\alb^n)/n \in \overline C_1$. Then $\alb$ is a relative interior point of $F$, so that it does not belong to any another facet of $\overline C_1$. Thus, $\alb$ is a solution of the following system:
\begin{equation}
\begin{cases}
\sum_{i\in E_j} x_i  < 1 & \text{ for } j = 1,\ldots,r,\\
\sum_{i\in E_j} x_i  \geqslant  1 & \text{ for } j = r+1,\ldots,m,\\
x_1\geqslant 0,\ldots,x_n\geqslant 0.
\end{cases}
\end{equation}
Therefore, $n\alb$ is a solution of the following system
\begin{equation}
\begin{cases}
\sum_{i\in E_j} x_i  < n & \text{ for } j = 1,\ldots,r,\\
\sum_{i\in E_j} x_i  \geqslant  n & \text{ for } j = r+1,\ldots,m,\\
x_1\geqslant 0,\ldots,x_n\geqslant 0.
\end{cases}
\end{equation}
Together with the fact that $n\alb\in \N^n$, it yields $n\alb \in \Omega _n$. Thus, $\Omega _n\cap\N^n\ne\emptyset$.
\end{proof}

We are now ready to prove the first main result of this section. Without loss of generality we may assume that $\E\ne\emptyset$ and thus $J(\H)\ne 0.$ 

\begin{thm} \label{T1} Let $\H$ be a balanced hypergraph. Then the depth funtion of $J(\H)$ has non-increasing property.
\end{thm}
\begin{proof}  Fix any $t\geqslant 1$. We need to show that 
$$\depth R/J(\H)^t\geqslant \depth R/J(\H)^{t+1}.$$

Indeed, let $d:=\depth R/J(\H)^t$. We have $H_{\mi}^d(R/J(\H)^t)_{\alb} \ne\zv$ for some $\alb\in\Z^n$. By Equation $(\ref{TA})$ we have
\begin{equation}\label{N11}
\dim_K \h_{d-|G_{\alb}|-1}(\Delta_{\alb}(J(\H)^t);K) =\dim_K H_{\mi}^d(R/J(\H)^t)_{\alb} \ne 0.
\end{equation}
In particular, $\Delta_{\alb}(J(\H)^t)$ is not acyclic.

If $G_{\alb} =\{1,\ldots,n\}$, then $\Delta_{\alb}(J(\H)^t)$ is either $\{\emptyset\}$ or a void complex. Since it is not acyclic, so $\Delta_{\alb}(J(\H)^t) =\{\emptyset\}$. But then by $(\ref{degree-complex})$ we would have $J(\H) = 0$, a contradiction. 

Therefore, $G_{\alb} \ne \{1,\ldots,n\}$.
We may assume that $G_{\alb} =\{p+1,\ldots,n\}$ for some $1\leqslant p\leqslant n$. Set $S := K[x_1,\ldots,x_p]$. Let $\H'$ be the subhypergraph of $\H$ on the vertex set $\V' =\{1,\ldots,p\}$ with the edge set $\E' =\{E\in \E \mid E \subseteq \V'\}$. By \cite[Proposition 4.3] {AB}, since $\H$ is balanced, so is  $\H'.$

Moreover, by $(\ref{intersect})$ we obtain:
\begin{equation}\label{EQ-LOCALIZATION1}
J(\H)R_{G_{\alb}} \cap S = J(\H').
\end{equation}

Let $\alb^* := (\alb_1,\ldots,\alb_p) \in \N^p$. By using Formulas $(\ref{degree-complex})$ and $(\ref{EQ-LOCALIZATION1})$ we get
\begin{equation}\label{POWER1}
\Delta_{\alb^*}(J(\H')^t) = \Delta_{\alb}(J(\H)^t) \text{ for any } t\geqslant 1.
\end{equation}
Together with $(\ref{N11})$, it gives $\h_{d-|G_{\alb}|-1}(\Delta_{\alb^*}(J(\H')^t);K)\ne 0$. In particular, the complex $\Delta_{\alb^*}(J(\H')^t);K)$ is not acyclic.

Suppose that $\E' =\{E_1,\ldots,E_v\}$ where $v\geqslant 1$. Then by Equation $(\ref{complex-cover})$ $$\Delta(J(\H')) = \left<\V'\setminus E_1,\ldots, \V'\setminus E_v\right>.$$
By Lemma $\ref{uni-complex}$ we may assume that 
$$\Delta_{\alb^*}(J(\H')^t) =\left<\V'\setminus E_1,\ldots, \V'\setminus E_q\right>,$$
with $1\leqslant q\leqslant v$. 

For each integer $s\geqslant 1$, let $\Omega _s$ be the set of solutions in $\R^p$ of the following system:
\begin{equation}\label{Qt}
\begin{cases}
\sum_{i\in E_j} x_i  \leqslant s-1 & \text{ for } j = 1,\ldots,q,\\
\sum_{i\in E_j} x_i  \geqslant  s & \text{ for } j = q+1,\ldots,v,\\
x_1\geqslant 0,\ldots,x_p\geqslant 0.
\end{cases}
\end{equation}
Then $\alb^* \in \Omega _t\cap N^p$, so $\Omega _t\cap\N^p\ne \emptyset$. By Lemma $\ref{polytope_1}$ we have $\Omega _{t+1}\cap \N^p\neq\emptyset$. Let $\btb=(\beta _1,\ldots,\beta_p)\in\Omega_{t+1}\cap\N^p$. Then, 
\begin{equation}\label{Q(t+1)}
\begin{cases}
\sum_{i\in E_j} \btb_i  \leqslant t & \text{ for } j = 1,\ldots,q,\\
\sum_{i\in E_j} \btb _i  \geqslant  t+1& \text{ for } j = q+1,\ldots,v,\\
\end{cases}
\end{equation}
It follows that
\begin{equation}\label{EQ0001}
\Delta_{\btb }(J(\H')^{t+1}) =\Delta_{\alb^*}(J(\H')^t) =\left<\V'\setminus E_1,\ldots, \V'\setminus E_q\right>
\end{equation}

Let $\btb' = (\btb _1,\ldots,\btb _p,-1,\ldots,-1)\in \Z^n$. Then $G_{\btb '} = G_{\alb}$, and by $(\ref{POWER1})$ and $(\ref{EQ0001})$ we obtain
$$\Delta_{\btb '}(J(\H)^{t+1}) = \Delta_{\btb }(J(\H')^{t+1}) = \Delta_{\alb^*}(J(\H')^t) =  \Delta_{\alb}(J(\H)^t).$$
Again, by Lemma $\ref{TA}$ we have
\begin{align*}
\dim_K H_{\mi}^d(R/J(\H)^{t+1})_{\btb' } &= \dim_K \h_{i-|G_{\btb' }|-1}(\Delta_{\btb' }(J(\H)^{t+1});K)\\
&=\dim_K \h_{i-|G_{\alb}|-1}(\Delta_{\alb}(J(\H)^t);K)\ne 0.
\end{align*}
Consequently, $H_{\mi}^p(R/J(\H)^{t+1})\ne 0$, and so
$$\depth R/J(\H)^{t+1} \leqslant d.$$
Therefore, the proof of the theorem is complete.
\end{proof}

In the case $\H$ is a bipartite graph. It is easy to see that by Theorem $\ref{T1}$ we receive \cite[Theorem $3.2$]{CPSTY}.

\begin{cor} If $G$ is  bipartite graph. Then $J(G)$ has non-increasing depth function.
\end{cor}

Next, we give an upper bound for $\dstab(J(\H))$ where $\H$ is a balanced hypergraph. The our second main result is the following.

\begin{thm} \label{T2} Let $\H = (\V,\E)$ be a balanced hypergraph with the vertex set\break $\V =\{1,\ldots,n\}$. Then
$$\depth R/J(\H)^t = n -\ell(J(\H)) \text{ for all } t\geqslant n.$$
Moreover, $\dstab(J(\H)) \leqslant n$.
\end{thm}
\begin{proof}  Since $\H$ is balanced hypergraph, then $J(\H)$ is totally torsion-free\break by \cite[Theorem 1.4]{HHTZ}. Therefore, by \cite[Proposition $10.3.2$ and Theorem $10.3.13$]{HH1} we have
$$\underset{t\rightarrow \infty}{\lim}\depth R/J(\H)^t = \dim R-\ell(J(\H)).$$
Together with Theorem $\ref{T1}$, this yields
\begin{equation}\label{dstab-H}
\dstab(J(\H)) =\min\{t\geqslant 1\mid \depth R/J(\H)^t = \dim R -\ell(J(\H)\}.
\end{equation}
Hence, it suffices to show that $\dstab(J(\H)) \leqslant n$.  

Let $s:=\dstab(J(\H))$ and $d:= n-\ell(J(\H))$. Then, $H_{\mi}^d(R/J(\H)^s)_{\alb} \ne\zv$ for some $\alb\in\Z^n$. By Equation $(\ref{TA})$ we have
\begin{equation}\label{N1}
\dim_K \h_{d-|G_{\alb}|-1}(\Delta_{\alb}(J(\H)^s);k) =\dim_K H_{\mi}^d(R/J(\H)^s)_{\alb} \ne 0.
\end{equation}
In particular, $\Delta_{\alb}(J(\H)^s)$ is not acyclic.

Let $F:=G_{\alb}$, we now can consider two cases:

\medskip

{\it Case $1$}: If $F =\{1,\ldots,n\}$, then $\Delta_{\alb}(J(\H)^s)$ is either $\{\emptyset\}$ or a void complex. Since it is not acyclic, so $\Delta_{\alb}(J(\H)^s) =\{\emptyset\}$. But then by $(\ref{degree-complex})$ we would have $J(\H) = 0$, a contradiction.

\medskip

{\it Case $2$}: $F \ne \{1,\ldots,n\}$.
We may assume that $F =\{p+1,\ldots,n\}$ for some $1\leqslant p\leqslant n$. Set $S := K[x_1,\ldots,x_p]$. Let $\H'$ be the subhypergraph of $\H$ on the vertex set $\V' =\{1,\ldots,p\}$ with the edge set $\E' =\{E\in \E \mid E \subseteq \V'\}$. By \cite[Proposition 4.3] {AB}, since $\H$ is balanced,  $\H'$ is balanced too.

Moreover, by $(\ref{intersect})$ we obtain:
\begin{equation}\label{EQ-LOCALIZATION}
J(\H)R_{G_{\alb}} \cap S = J(\H').
\end{equation}
Let $\alb^* := (\alb_1,\ldots,\alb_p) \in \N^p$. By using Formulas $(\ref{degree-complex})$ and $(\ref{EQ-LOCALIZATION})$ we get
\begin{equation}\label{POWER}
\Delta_{\alb^*}(J(\H')^s) = \Delta_{\alb}(J(\H)^s) \text{ for any } s\geqslant 1.
\end{equation}
Together with $(\ref{N1})$, it gives $\h_{d-|G_{\alb}|-1}(\Delta_{\alb^*}(J(\H')^s);K)\ne 0$. It implies that\break
 $\Delta_{\alb^*}(J(\H')^s);K)$ is also not acyclic.

Suppose that $\E' =\{E_1,\ldots,E_k\}$ where $k\geqslant 1$. Then, by Equation $(\ref{complex-cover})$ $$\Delta(J(\H')) = \left<\V'\setminus E_1,\ldots, \V'\setminus E_k\right>.$$
By Lemma $\ref{uni-complex}$ we may assume that 
$$\Delta_{\alb^*}(J(\H')^t) =\left<\V'\setminus E_1,\ldots, \V'\setminus E_q\right>$$
with $1\leqslant q\leqslant k$. 

For each integer $l\geqslant 1$, let $\Omega _l$ be the set of solutions in $\R^p$ of the following system:
\begin{equation}\label{Qtt}
\begin{cases}
\sum_{i\in E_j} x_i  \leqslant l-1 & \text{ for } j = 1,\ldots,q,\\
\sum_{i\in E_j} x_i  \geqslant  l & \text{ for } j = q+1,\ldots,k,\\
x_1\geqslant 0,\ldots,x_p\geqslant 0.
\end{cases}
\end{equation}
Then $\alb^* \in \Omega _s$ by Lemma $\ref{uni-complex}$, so $\Omega _s\ne \emptyset$. 

By Lemma $\ref{polytope_1}$ we have $\Omega _{p}\cap \N^p\neq\emptyset$. Since $p \leqslant n$, by Lemma $\ref{polytope_1}$ again we have $\Omega _{p}\cap \N^n\neq\emptyset$.

Let $\gmb \in \Omega _p\cap \N^n$. Then $\gmb$ satisfies the system $(\ref{Qtt})$ by replacing $l$ by $n$. Together with Lemma $\ref{uni-complex}$ we have 
\begin{equation}\label{POWER2}
\Delta_{\gmb}(J(\H')^n) = \Delta_{\alb^*}(J(\H')^s)= \{\V\setminus E_1,\ldots, \V \setminus E_k\}.
\end{equation}
Let $\gmb' = (\gmb _1,\ldots,\gmb _p,-1,\ldots,-1)\in \Z^n$. Then $G_{\gmb '} = G_{\alb}$, and by $(\ref{POWER})$ and $(\ref{POWER2})$ we obtain
$$\Delta_{\gmb '}(J(\H)^{n}) = \Delta_{\gmb }(J(\H')^{n}) = \Delta_{\alb^*}(J(\H')^s) =  \Delta_{\alb}(J(\H)^s).$$

By Equation $(\ref{TA})$ we yields
\begin{align*}
\dim_K H_{\mi}^d(R/J(\H)^{n})_{\gmb' } &= \dim_K \h_{i-|G_{\gmb' }|-1}(\Delta_{\gmb' }(J(\H)^{n});K)\\
&=\dim_K \h_{i-|G_{\alb}|-1}(\Delta_{\alb}(J(\H)^s);K)\ne 0.
\end{align*}
In particular, $\depth R/J(\H)^n \leqslant d$. On the other hand, by Theorem $\ref{T1}$ we have
$$\depth R/J(\H)^n \geqslant \lim_{l\to\infty} \depth R/J(\H)^l = d.$$
It follows that  $\depth R/J(\H)^n = d$, so $s\leqslant n$ by $(\ref{dstab-H})$, as required. 
\end{proof}

\subsection*{Acknowledgment}  This work is partially supported by NAFOSTED (Vietnam) under the grant number 101.04-2015.02. I am also partially supported by Thai Nguyen University of Sciences under the grant number \DJ H2016-TN06-03.

\end{document}